\documentclass[11pt,a4paper]{article}
\usepackage[letterpaper,top=2.7cm,bottom=2.7cm,left=2.5cm,right=2.5cm,marginparwidth=1.75cm]{geometry}
\DeclareMathAlphabet{\mathpzc}{OT1}{pzc}{m}{it}
\usepackage{amsmath,amsfonts,amsthm,amssymb}
\usepackage{amsmath,amssymb,color}
\usepackage{graphicx,fancyhdr}
\usepackage{mathrsfs,float}
\usepackage{multirow}
\usepackage{dsfont}
\usepackage{subfigure}
\usepackage{booktabs}
\usepackage{yfonts}
\usepackage{natbib}
\usepackage{authblk}
\usepackage{caption}
\usepackage{geometry}
\usepackage{bbm}
\makeatletter
\usepackage{hyperref}

\numberwithin{equation}{section}
\newtheorem{thm}{Theorem}[section]
\newtheorem{rem}[thm]{Remark}

\newtheorem{lem}[thm]{Lemma}

\title{Uniqueness of an Inverse Coefficient Problem for a Time-Fractional Damped Wave Equation from Boundary Measurements}
\author[a]{Kai Yu \thanks{Corresponding author: yukaimailbox@163.com}}
\affil[a]{School of Mathematics and Statistics, Ningbo University, Ningbo 315211, China}

\date{}

\begin{document}
\maketitle

\begin{abstract}

This paper studies an inverse coefficient problem for a time-fractional damped wave equation on a finite time interval. The aim is to determine two spatially varying coefficients, namely the fractional damping coefficient and the zeroth-order potential, from the associated Dirichlet-to-Neumann (DtN) map. We first prove the well-posedness of the forward problem for boundary data with sufficient regularity to admit a suitable lifting. The main result is a uniqueness theorem showing that if two coefficient pairs give rise to the same DtN map, then the corresponding coefficients coincide almost everywhere in the domain. The proof is based on a convolution-type integral identity that eliminates the unknown boundary traces, the construction of high-frequency beam solutions adapted to the singular kernel of the Caputo derivative, and a detailed asymptotic analysis of the resulting fractional convolution terms. This result extends classical uniqueness results for hyperbolic inverse problems to a fractional-order model with damping and provides a theoretical basis for related applications involving memory and viscoelastic effects.


\end{abstract}


\section{Introduction}

Inverse coefficient problems for partial differential equations concern the determination of spatially varying parameters, such as conductivity, density, damping, or absorption coefficients, from boundary measurements, typically encoded in the DtN map. Such problems have attracted extensive attention due to their important applications in non-destructive testing, medical imaging, geophysics, and material science~\cite{Isakov2006,KK2015,Uhlmann2009,Z2002}. The study of this class of problems goes back to Calderón's seminal work~\cite{Calderon1980} on the inverse conductivity problem, now known as the Calderón problem. For elliptic equations, Sylvester and Uhlmann~\cite{SU1987} established global uniqueness by constructing complex geometric optics (CGO) solutions, while Alessandrini~\cite{Alessandrini1988} obtained conditional logarithmic stability estimates. For hyperbolic equations, global uniqueness results were proved in~\cite{Bukhgeim1981} by means of Carleman estimates. Rakesh and Symes~\cite{Rakesh1988} studied the recovery of time-dependent potentials in wave equations from the DtN map and proved uniqueness using beam-type solutions adapted to the hyperbolic setting. These foundational works have played a central role in the development of inverse problems for hyperbolic equations, which are closely related to imaging in elastic media and non-destructive evaluation.

Classical models based on integer-order partial differential equations, however, may be insufficient for describing memory and viscoelastic effects observed in complex media such as biological tissues, polymers, and composite materials. For wave propagation in such media, power-law frequency attenuation is often observed experimentally, indicating the presence of temporal nonlocality that can be described by fractional derivatives or memory kernels~\cite{CH2003}. This motivates the use of time-fractional damping, where the Caputo derivative $\partial_t^\alpha$, $\alpha\in(0,1)$, provides an effective model for power-law attenuation and memory effects~\cite{Mai2022,Subdiffusive2000,Metzler2000}. The resulting nonlocal term substantially changes the analytic structure of the governing equation and makes many classical methods difficult to apply directly. Motivated by these considerations, we study the following initial-boundary value problem for a time-fractional damped wave equation posed on a bounded domain $\Omega\subset\mathbb{R}^n$ with smooth boundary $\partial\Omega$:
\begin{equation}\label{eq1.1}
\left\{\begin{array}{ll}
\partial_{t}^{2} u(x, t)-\Delta u(x, t)+p(x)\partial_{t}^{\alpha}u(x,t)+q(x)u(x,t)=0, & (x,t) \in Q,   \\
u(x, t)=g(x,t), &  (x, t)\in\Sigma,\\
u(x, 0)=0,~\partial_t u(x, 0)=0, & x \in {\Omega}, 
\end{array}\right.
\end{equation}
where $Q:=\Omega\times(0,T)$ and $\Sigma:=\partial\Omega\times(0,T)$ denote the space-time cylinder and its lateral boundary, respectively. Here $\partial_t^\alpha$ is the Caputo fractional derivative of order $\alpha\in(0,1)$ with respect to $t$, whose definition is recalled in Section~\ref{sec2}. The coefficients $p(x)$ and $q(x)$ represent, respectively, the fractional damping coefficient and the acoustic absorption potential. The inverse problem considered in this paper is to determine these two spatially dependent coefficients from the corresponding DtN map.

The equation \eqref{eq1.1} features both a time-fractional derivative, modeling memory and damping effects, and a hyperbolic (wave-type) spatial operator, thereby linking it to two active lines of research: inverse problems for fractional evolution equations and inverse problems for hyperbolic equations. Inverse problems for fractional evolution equations have received increasing attention in recent years~\cite{Cheng2011,Jin2015,Li2015,Li2019,Miller2013}. In the hyperbolic setting, the classical wave equation, corresponding formally to the case without fractional damping, and the standard damped wave equation have been studied extensively in~\cite{Rakesh1988,Bel1987,Kli1992,KL2019,Sakamoto2011,Kian2021} and~\cite{Isa1991,Rom2020,Kian2016}, respectively. For hyperbolic equations with memory terms, Bukhgeim, Dyatlov and Isakov~\cite{Bukhgeim1997} used the Fourier transform in time to obtain logarithmic stability estimates on an infinite time interval. This uniqueness result was later extended to finite time intervals~\cite{Dya2003} and to partial boundary measurements~\cite{BDU2001}. Related identification problems for memory kernels were further investigated in~\cite{JW2001,Col2007,DS2015}. More recently, Kaltenbacher and Rundell~\cite{Kal2022} studied fractional wave equations and clarified the interaction between the fractional order and the damping profile, while the monograph~\cite{Sel2016} provides a comprehensive overview of fractional models and related inverse problems.

In this paper, we develop a geometric-optics approach adapted to the singular kernel of the Caputo derivative. We prove that the spatial coefficients $p(x)$ and $q(x)$ in \eqref{eq1.1} are uniquely determined by the DtN map on a finite time interval, provided that $T$ larger than the diameter of $\Omega$, and that the coefficients belong to $L^\infty(\Omega)$. The proof begins with a well-posedness result for the forward problem. For sufficiently smooth boundary data, a lifting argument reduces the problem to one with homogeneous boundary condition and a suitable volume source term. We then derive an integral identity from the equality of the DtN maps, which eliminates unknown boundary traces and involves only volume integrals containing the differences of the coefficients.

The main part of the proof is the construction of high-frequency beam solutions adapted to the fractional operator. More precisely, oscillatory plane-wave type functions multiplied by suitable cut-off functions are shown to approximate exact solutions, with an $L^2(Q)$ remainder of order $\sigma^{\alpha-1}$ as the large parameter $\sigma\to\infty$. The estimates rely on a careful treatment of the singular kernel
\[
k_\alpha(t)=\frac{t^{-\alpha}}{\Gamma(1-\alpha)}
\]
appearing in the Caputo derivative, together with an appropriate splitting of the time convolution and integration-by-parts arguments for the oscillatory terms. Substituting these special solutions into the integral identity and analyzing the leading-order asymptotics yield algebraic relations involving the Fourier transforms of the coefficient differences. Letting the asymptotic parameter tend to infinity shows that these Fourier transforms vanish identically. Consequently, by the Fourier uniqueness theorem, the coefficient differences vanish almost everywhere in $\Omega$, which proves the desired uniqueness result.

The paper is organized as follows. Section~\ref{sec2} introduces the functional setting and notation, recalls basic properties of fractional derivatives, presents the main results, and proves the well-posedness theorem for the forward problem. It also formulates the uniqueness theorem for the inverse problem. Section~\ref{sec3} is devoted to the proof of the uniqueness theorem, based on an integral identity, the construction of high-frequency beam solutions, and the asymptotic analysis of the associated fractional convolution terms. Finally, Section~\ref{sec4} contains concluding remarks and discusses possible directions for future work.


\section{Preliminaries and main results}\label{sec2}

We begin by fixing the notation and recalling several basic properties of fractional derivatives that will be used throughout the paper. We then formulate the forward problem, prove its well-posedness, define the Dirichlet-to-Neumann map, and finally state the main uniqueness result.

\subsection*{Preliminaries}

We first introduce some notation and function spaces. The $L^2(\Omega)$ inner product is denoted by $(\cdot,\cdot)$, and the associated norm by $\|\cdot\|$. We write \(L^2(Q):=L^2(0,T;L^2(\Omega))\),
and denote by $C([0,T];X)$ the space of continuous functions from $[0,T]$ into a Banach space $X$. The space $H_0^1(\Omega)$ is equipped with the equivalent norm $\|\nabla\cdot\|_{L^2(\Omega)}$, which is equivalent to the standard $H^1$ norm by Poincar\'e's inequality. Its dual space is denoted by $H^{-1}(\Omega)$. For standard facts on Sobolev spaces, we refer to~\cite{AF2003}.

Convolution in time is denoted by $*$. For measurable functions $f,g:(0,T)\to\mathbb{R}$, we set
\[
(f*g)(t):=\int_0^t f(t-\tau)g(\tau)\,d\tau,
\qquad t\in(0,T),
\]
whenever the integral is well defined. In particular, if $f\in L^1(0,T)$ and $g\in L^2(0,T)$, then Young's inequality gives
\[
\|f*g\|_{L^2(0,T)}
\leq
\|f\|_{L^1(0,T)}\|g\|_{L^2(0,T)}.
\]
For space-time functions, convolution is always understood with respect to the time variable; that is,
\[
(f*g)(x,t):=\int_0^t f(x,t-\tau)g(x,\tau)\,d\tau
\]
whenever the expression is meaningful.

For $\alpha\in(0,1)$, the Caputo fractional derivative and the Riemann--Liouville fractional integral of order $1-\alpha$ are defined respectively by
\[
\partial_t^\alpha u(t)
=
\frac{1}{\Gamma(1-\alpha)}
\int_0^t
\frac{u'(\tau)}{(t-\tau)^\alpha}\,d\tau,
\qquad
I^{1-\alpha}u(t)
=
\frac{1}{\Gamma(1-\alpha)}
\int_0^t
\frac{u(\tau)}{(t-\tau)^\alpha}\,d\tau,
\]
whenever the integrals are well defined. Setting \(k_\alpha(t):=\frac{t^{-\alpha}}{\Gamma(1-\alpha)}
\), we have the convolution representations
\[
\partial_t^\alpha u=k_\alpha*u',
\qquad
I^{1-\alpha}u=k_\alpha*u.
\]
In particular, if $u\in H^1(0,T)$ and $u(0)=0$, then the Caputo derivative coincides with the corresponding Riemann--Liouville fractional derivative.

The kernel $k_\alpha$ is positive, decreasing, and belongs to $L^1(0,T)$. Moreover,
\[
\|k_\alpha\|_{L^1(0,T)}
=
\frac{T^{1-\alpha}}{\Gamma(2-\alpha)}.
\]
Consequently, Young's convolution inequality implies that $I^{1-\alpha}$ is bounded on $L^2(0,T)$. In particular, for every $u\in H^1(0,T)$ with $u(0)=0$,
\[
\|\partial_t^\alpha u\|_{L^2(0,T)}
\leq
\|k_\alpha\|_{L^1(0,T)}
\|\partial_t u\|_{L^2(0,T)}.
\]

\subsection*{Well-posedness of the forward problem}

We now formulate the forward initial-boundary value problem \eqref{eq1.1}, where \(p,q\in L^\infty(\Omega)\). The boundary data \(g(x,t)\) are taken from
\[
\mathcal{B}
:=
L^2(0,T;H^{3/2}(\partial\Omega))
\cap
H^2(0,T;H^{1/2}(\partial\Omega)),
\]
and we impose the compatibility conditions
\begin{equation}\label{eq com con}
g(\cdot,0)=\partial_t g(\cdot,0)=0
\quad \text{on } \partial\Omega.    
\end{equation}
This choice of $\mathcal{B}$ ensures the existence of a lifting with sufficient spatial and temporal regularity, allowing us to reduce \eqref{eq1.1} to a homogeneous boundary value problem with a volume source term in $L^2(Q)$.  With the functional framework and the assumptions on the data in place, we can now state the well‑posedness result for the forward problem. 

\begin{thm}[Well-posedness]\label{thm wellp}
Let $p,q\in L^\infty(\Omega)$ and $g\in\mathcal{B}$ satisfy the compatibility conditions \eqref{eq com con}. Then problem \eqref{eq1.1} admits a unique solution
\[
u\in C([0,T];H^1(\Omega))
\cap
C^1([0,T];L^2(\Omega)).
\]
Moreover, there exists a constant $C>0$, depending only on $\Omega$, $T$, $\alpha$, $\|p\|_{L^\infty(\Omega)}$, and $\|q\|_{L^\infty(\Omega)}$, such that
\[
\|u\|_{C([0,T];H^1(\Omega))}
+
\|\partial_t u\|_{C([0,T];L^2(\Omega))}
\leq
C\|g\|_{\mathcal{B}}.
\]
\end{thm}

\begin{proof}[\bf{Proof of Theorem \ref{thm wellp}}]

Since $g\in \mathcal{B}$ with $g(\cdot,0)=\partial_t g(\cdot,0)=0$, by standard trace and extension theorems for Sobolev spaces
(see \cite{Lions1972}), there exists a lifting function $\tilde{g} = \mathcal{R}g$ such that
\[
\tilde{g}|_{\partial\Omega\times(0,T)} = g,\quad \tilde{g}(\cdot,0) = \partial_t \tilde{g}(\cdot,0) = 0,
\]
and
\[
\tilde{g} \in L^2(0,T;H^2(\Omega)) \cap H^2(0,T;H^{1}(\Omega))
\]
with the estimate
\[
\|\tilde{g}\|_{L^2(0,T;H^2(\Omega))} + \|\tilde{g}\|_{H^2(0,T;H^{1}(\Omega))} \le C \|g\|_{\mathcal{B}}.
\]
Set \(v = u - \tilde{g}\). Then \(v\) satisfies homogeneous boundary and initial conditions, and solves
\begin{equation}\label{eq2.1}
\partial_t^2 v-\Delta v+p(x)\partial_t^\alpha v+q(x)v=F(x,t),
\qquad (x,t)\in Q,
\end{equation}
with source term
\[
F := -\partial_t^2 \tilde{g} + \Delta \tilde{g} - p\,\partial_t^\alpha \tilde{g} - q\tilde{g}.
\]
Thanks to the regularity of $\tilde{g}$ and the boundedness of $p, q$, together with the estimate for the Caputo derivative \(\|\partial_t^\alpha \tilde g\|_{L^2(Q)}
\le C\|\partial_t\tilde g\|_{L^2(Q)}\) (see \cite{Podlubny1999}), we conclude that $F \in L^2(Q)$ and
\begin{equation}\label{eq2.2}
\|F\|_{L^2(Q)} \le C \|g\|_{\mathcal{B}},
\end{equation}
where \(C>0\) depends only on \(\Omega\), \(T\), \(\alpha\), \(\|p\|_{L^\infty(\Omega)}\), and \(\|q\|_{L^\infty(\Omega)}\).

The homogeneous initial-boundary value problem \eqref{eq2.1} falls within the framework of~\cite{HK2023}. By~\cite[Theorem~1.1]{HK2023}, there exists a unique solution
\[
v\in C([0,T];H_0^1(\Omega))\cap C^1([0,T];L^2(\Omega))
\]
satisfying the stability estimate
\[
\|v\|_{C([0,T];H_0^1(\Omega))}
+
\|\partial_t v\|_{C([0,T];L^2(\Omega))}
\le C\|F\|_{L^2(Q)}.
\]
Here \(C>0\) depends only on \(\Omega\), \(T\), \(\alpha\), \(\|p\|_{L^\infty(\Omega)}\), and \(\|q\|_{L^\infty(\Omega)}\).

Returning to the original variable \(u = v + \tilde{g}\) and using the estimates for \(v\) together with \eqref{eq2.2}, we obtain
\[
\begin{aligned}
\|u\|_{C([0,T];H^1(\Omega))}
&\le \|v\|_{C([0,T];H_0^1(\Omega))} + \|\tilde{g}\|_{C([0,T];H^1(\Omega))}
\le C\|F\|_{L^2(Q)} + C\|g\|_{\mathcal{B}}
\le C\|g\|_{\mathcal{B}},
\end{aligned}
\]
and analogously
\[
\begin{aligned}
\|\partial_t u\|_{C([0,T];L^2(\Omega))}
&\le \|\partial_t v\|_{C([0,T];L^2(\Omega))} + \|\partial_t \tilde{g}\|_{C([0,T];L^2(\Omega))}
\le C\|F\|_{L^2(Q)} + C\|g\|_{\mathcal{B}}
\le C\|g\|_{\mathcal{B}}.
\end{aligned}
\]
The regularity \(u \in C([0,T];H^1(\Omega)) \cap C^1([0,T];L^2(\Omega))\) follows directly from that of \(v\) and \(\tilde{g}\). The uniqueness of \(u\) follows from the uniqueness of \(v\) for the homogeneous boundary problem. This completes the proof.
\end{proof}

\subsection*{Dirichlet-to-Neumann map and uniqueness result}

The well-posedness result above allows us to define the boundary measurement operator. For $p,q\in L^\infty(\Omega)$, let $u$ be the solution of \eqref{eq1.1} corresponding to the boundary input $g\in\mathcal{B}$. The Dirichlet-to-Neumann map is defined by
\[
\Lambda_{p,q}:\mathcal{B}\ni g
\longmapsto
\partial_\nu u|_\Sigma,
\]
where $\partial_\nu$ denotes the outward normal derivative on $\partial\Omega$. The main inverse problem of this paper is to determine the two spatial coefficients $p$ and $q$ from $\Lambda_{p,q}$. Our uniqueness result is as follows.

\begin{thm}[Uniqueness]\label{thm uniqueness}
Let $p_j,q_j\in L^\infty(\Omega)$, $j=1,2$, and let $T>2r$, where $r$ is the radius of the smallest closed ball containing $\overline{\Omega}$ with the origin placed at its centre. If
\[
\Lambda_{p_1,q_1}(g)=\Lambda_{p_2,q_2}(g)
\qquad
\forall\,g\in\mathcal{B},
\]
then
\[
p_1=p_2,
\qquad
q_1=q_2
\quad
\text{a.e. in } \Omega.
\]
\end{thm}

\begin{rem}
The condition $T>2r$ is essential for the geometric-optics construction used in the uniqueness proof; it guarantees that the beam solutions are fully supported inside $\Omega$ for large times. No smoothness or sign condition is required on the coefficients beyond $L^\infty(\Omega)$.
\end{rem}


\section{Uniqueness for the inverse problem}\label{sec3}

In this section we prove Theorem~\ref{thm uniqueness}. The proof consists of three main ingredients. First, we derive an integral identity from the equality of the DtN maps. Second, we construct high-frequency beam solutions adapted to the fractional damping term and estimate the corresponding remainders. Finally, we substitute these special solutions into the integral identity and use their high-frequency asymptotics to recover the Fourier transforms of the coefficient differences. Throughout this section, for two sets of coefficients $(p_j,q_j)$, $j=1,2$, we write \(p:=p_1-p_2\) and \(q:=q_1-q_2\).

\begin{lem}[Integral identity]\label{lem iden}

Let $u$ be the solution of \eqref{eq1.1} corresponding to the coefficients $(p_1,q_1)$ and boundary input $g_1$, and  $v$ be the solution corresponding to the coefficients $(p_2,q_2)$ and boundary input $g_2$. Assume that \(\Lambda_{p_1,q_1}=\Lambda_{p_2,q_2}\). Then, for every $t\in(0,T]$, one has
\begin{equation}\label{eq iduv0}
\int_{\Omega}p\,k_\alpha*\partial_t(u*v)\,dx
+
\int_{\Omega}q\,u*v\,dx
=0.
\end{equation}
\end{lem}

\begin{proof}
The solutions $u$ and $v$ satisfy
\[
\partial_t^2u-\Delta u+p_1\partial_t^\alpha u+q_1u=0,\quad
\partial_t^2v-\Delta v+p_2\partial_t^\alpha v+q_2v=0.
\]
Convolve the first equation with $v$ and the second with $u$ in time, then subtract to obtain
\begin{equation}\label{eq iduv}
\big[(\partial_t^2 u)*v - (\partial_t^2 v)*u\big] 
- \big[(\Delta u)*v - (\Delta v)*u\big] + \big[p_1 (\partial_t^\alpha u)*v - p_2 (\partial_t^\alpha v)*u\big] 
+ (q_1-q_2)\,u*v = 0. 
\end{equation}
Since $u$ and $v$ have zero initial displacement and velocity, we have
\[
(\partial_t^2u)*v=(\partial_t^2v)*u=\partial_t^2(u*v).
\]
Thus the first bracket in \eqref{eq iduv} vanishes. 

For the elliptic part, using the product rule with respect to the spatial variables gives
\[
(\Delta u)*v-(\Delta v)*u
=
\nabla\cdot\bigl((\nabla u)*v-(\nabla v)*u\bigr).
\]
Moreover, since
\[
\partial_t^\alpha w=k_\alpha*\partial_tw
\]
for functions satisfying $w(\cdot,0)=0$, the associativity and commutativity of time convolution imply
\[
(\partial_t^\alpha u)*v
=
(k_\alpha*\partial_tu)*v
=
k_\alpha*(\partial_tu*v)
=
k_\alpha*\partial_t(u*v),
\]
where we used $u(\cdot,0)=0$. Similarly,
\[
(\partial_t^\alpha v)*u
=
k_\alpha*\partial_t(u*v).
\]
Therefore,
\[
p_1(\partial_t^\alpha u)*v
-
p_2(\partial_t^\alpha v)*u
=
(p_1-p_2)k_\alpha*\partial_t(u*v).
\]
Substituting these identities into \eqref{eq iduv}, we get
\[
-\nabla\cdot\bigl((\nabla u)*v-(\nabla v)*u\bigr)
+
p\,k_\alpha*\partial_t(u*v)
+
q\,u*v
=0.
\]
Integrating over $\Omega$ and applying the divergence theorem yield
\begin{equation}\label{eq boundary_identity}
\begin{aligned}
\int_{\Omega}p\,k_\alpha*\partial_t(u*v)\,dx
+
\int_{\Omega}q\,u*v\,dx
=
\int_{\partial\Omega}
\left(
\partial_\nu u*v-\partial_\nu v*u
\right)dS .
\end{aligned}
\end{equation}

We now show that the boundary term on the right-hand side of
\eqref{eq boundary_identity} vanishes. Since
\[
u|_\Sigma=g_1,\qquad v|_\Sigma=g_2,
\]
we have
\[
\int_{\partial\Omega}
\left(
\partial_\nu u*v-\partial_\nu v*u
\right)dS
=
\int_{\partial\Omega}
\left(
\Lambda_{p_1,q_1}g_1*g_2
-
\Lambda_{p_2,q_2}g_2*g_1
\right)dS .
\]
By the assumption \(\Lambda_{p_1,q_1}=\Lambda_{p_2,q_2}\), hence the boundary term becomes
\[
\int_{\partial\Omega}
\left(
\Lambda_{p_2,q_2}g_1*g_2
-
\Lambda_{p_2,q_2}g_2*g_1
\right)dS .
\]

It remains to prove that this last expression is zero. To see this, consider the same argument as above, but now with the same coefficients
\((p_2,q_2)\) in both equations and with two boundary inputs \(g_1\) and \(g_2\). Let \(U\) and \(V\) be the corresponding solutions. Since the two coefficient pairs are identical, the volume terms involving the coefficient differences vanish. Therefore the boundary identity reduces to
\[
\int_{\partial\Omega}
\left(
\partial_\nu U*V-\partial_\nu V*U
\right)dS=0.
\]
Using \(U|_\Sigma=g_1\), \(V|_\Sigma=g_2\), and
\[
\partial_\nu U=\Lambda_{p_2,q_2}g_1,\qquad
\partial_\nu V=\Lambda_{p_2,q_2}g_2,
\]
we obtain
\[
\int_{\partial\Omega}
\left(
\Lambda_{p_2,q_2}g_1*g_2
-
\Lambda_{p_2,q_2}g_2*g_1
\right)dS=0.
\]
Consequently, the right-hand side of \eqref{eq boundary_identity} vanishes, and the integral identity
\eqref{eq iduv0} follows.
\end{proof}

\begin{lem}[High-frequency beam solutions]\label{lem cgo}
Let the origin be chosen as the centre of the smallest closed ball containing $\overline{\Omega}$, and let $r$ be its radius, so that \(|x|\le r, x\in\Omega\). Assume that \(T>2r\), we fix \(\epsilon>0\) sufficiently small such that \(0<\epsilon<\frac{T-2r}{2}\). This ensures that the interval \((2r+2\epsilon,T)\) is nonempty. Let $\theta_\epsilon\in C^\infty(\mathbb{R})$ satisfy
\[
\theta_\epsilon(s)=0\,\, (s\le0),\quad
\theta_\epsilon(s)=1\,\, (s\ge\epsilon),\quad
\theta_\epsilon'(s)\ge0.
\]
For a parameter \(\sigma>0\) and a unit vector \(\omega\in\mathbb{S}^{n-1}\), define the beam solution
\[
u_{\rm app}(x,t):=e^{i\sigma\phi(x,t)}A(x,t),\quad \phi(x,t):=x\cdot\omega+t,\quad
A(x,t):=\theta_\epsilon(\phi(x,t)-r).
\]
Then the boundary trace \(g:=u_{\rm app}|_{\Sigma}\) belongs to \(\mathcal{B}\) and satisfies the compatibility conditions \eqref{eq com con}. 

Moreover, let $u$ be the exact solution of \eqref{eq1.1} with boundary value $g$, coefficients $p,q\in L^\infty(\Omega)$, and zero initial conditions. Setting \(R:=u-u_{\rm app}\), we have
\begin{equation*}
\|R\|_{L^2(Q)}\le C\sigma^{\alpha-1},
\qquad
\|\partial_tR\|_{L^2(Q)}\le C\sigma^\alpha,
\end{equation*}
where $C>0$ depends on $\Omega$, $T$, $\epsilon$, $\alpha$, $\|p\|_{L^\infty(\Omega)}$, and $\|q\|_{L^\infty(\Omega)}$, but is independent of $\sigma$.
\end{lem}

\begin{proof}

The proof proceeds in several steps. We first verify the properties of the approximate solution \(u_{\mathrm{app}}\) (regularity of its boundary trace, zero initial conditions, and satisfaction of the unperturbed wave equation). Then we derive the equation for the remainder \(R = u - u_{\mathrm{app}}\) and reduce its estimate to that of the time-integrated source \(\mathcal{I}F\). Finally, we bound \(\mathcal{I}F\) and \(F\) using oscillatory integration by parts to obtain the desired estimates for \(\|R\|_{L^2(Q)}\) and \(\|\partial_t R\|_{L^2(Q)}\).

\medskip
\noindent\textbf{Step 1.  Regularity of the boundary trace, initial conditions, and the wave operator.} By construction, the boundary trace is
\[
g(x,t) = e^{i\sigma(x\cdot\omega + t)} \theta_\epsilon(x\cdot\omega + t - r), \quad x \in \partial\Omega, \, t\in(0,T),
\]
with time derivative
\[
\partial_t g(x,t) = e^{i\sigma(x\cdot\omega + t)}\big(i\sigma \theta_\epsilon(x\cdot\omega + t - r) + \theta_\epsilon'(x\cdot\omega + t - r)\big).
\]
For each fixed \(t \in [0,T]\), the map
\[
x \mapsto g(x,t) = \theta_\epsilon(x\cdot\omega + t - r)
\]
is smooth on \(\partial\Omega\). Moreover, since \(|x|\le r\) for all \(x\in\overline{\Omega}\) and \(|\omega|=1\), we have \(x\cdot\omega \le r\). Therefore
\[
x\cdot\omega + t - r \le 0 \implies \theta_\epsilon(x\cdot\omega + t - r) = 0,
\]
and
\[
x\cdot\omega + t - r \ge \epsilon \implies \theta_\epsilon(x\cdot\omega + t - r) = 1.
\]
Hence, for each \(t\in[0,T]\), the support of \(g(\cdot,t)\) is contained in the compact set
\[
\{ x \in \partial\Omega \,:\, r-t \le x\cdot\omega \le r-t+\epsilon \} \subset \partial\Omega,
\]
so that \(g(\cdot,t)\in C_c^\infty(\partial\Omega)\). Since \(\theta_\epsilon\) is smooth in \(t\) as well, the map \(t \mapsto g(\cdot,t)\) is \(C^\infty\), and therefore \(g \in C^\infty(\partial\Omega \times [0,T])\). Moreover, at \(t=0\) we have
\[
x\cdot\omega - r \le 0 \implies g(x,0) = \partial_t g(x,0) = 0,
\]
which verifies the vanishing initial condition required for the lifting argument.

Because \(\phi(x,0)=x\cdot\omega\le r\), we have \(\phi(x,0)-r\le0\), hence \(A(x,0)=0\) and \(A_t(x,0)=\theta_\epsilon'(\phi(x,0)-r)=0\). Consequently
\[
u_{\mathrm{app}}(x,0)=0,\quad 
\partial_t u_{\mathrm{app}}(x,0)=e^{i\sigma\phi(x,0)}(i\sigma A(x,0)+A_t(x,0))=0.
\]
Thus \(u_{\mathrm{app}}\) satisfies the same zero initial conditions as the exact solution.

Since \(A\) is a function of \(\phi - r\) only, we have
\[
A_t = \theta_\epsilon'(\phi - r), \quad \nabla A = \theta_\epsilon'(\phi - r)\omega, \quad A_{tt} = \theta_\epsilon''(\phi - r), \quad \Delta A = \theta_\epsilon''(\phi - r).
\]
Compute the derivatives:
\[
\partial_t u_{\mathrm{app}} = e^{i\sigma\phi}(i\sigma A + A_t),\qquad
\partial_t^2 u_{\mathrm{app}} = e^{i\sigma\phi}\big[ -\sigma^2 A + 2i\sigma A_t + A_{tt} \big],
\]
\[
\nabla u_{\mathrm{app}} = e^{i\sigma\phi}(i\sigma\omega A + \nabla A),\qquad
\Delta u_{\mathrm{app}} = \nabla\cdot\nabla u_{\mathrm{app}} = e^{i\sigma\phi}\big[ -\sigma^2 A + 2i\sigma A_t + \Delta A \big].
\]
Since \(A_{tt} = \Delta A\), we obtain
\[
\partial_t^2 u_{\mathrm{app}} - \Delta u_{\mathrm{app}} = e^{i\sigma\phi}(A_{tt} - \Delta A)=0.
\]

\medskip

\noindent\textbf{Step 2. Equation for the remainder and reduction via time integration.}
Let \(R = u - u_{\mathrm{app}}\). Then \(R\) satisfies zero initial and boundary conditions and
\begin{equation}\label{eq errid}
R_{tt} - \Delta R + p \partial_t^\alpha R + q R = -p \partial_t^\alpha u_{\mathrm{app}} - q u_{\mathrm{app}}:= F.
\end{equation}

Define the time-integrated remainder \(\mathcal{I}R(x,t) = \int_0^t R(x,\tau)\,d\tau\). Since \(R(x,0)=0\), we have \(R = \partial_t (\mathcal{I}R)\) and \(R_t = \partial_t^2 (\mathcal{I}R)\). Integrating the equation \eqref{eq errid} in time  and using the zero initial conditions gives
\[
R_t - \Delta (\mathcal{I}R) + p(k_\alpha * R) + q(\mathcal{I}R) = \mathcal{I}F,
\]
where we used \(\mathcal{I}(\partial_t^\alpha R) = k_\alpha * R\). Note that \(k_\alpha * \partial_t (\mathcal{I}R) = \partial_t^\alpha (\mathcal{I}R)\) because \(\mathcal{I}R(x,0) = 0\). Therefore,
\[
\partial_t^2 (\mathcal{I}R) - \Delta (\mathcal{I}R) + p \partial_t^\alpha (\mathcal{I}R) + q (\mathcal{I}R) = \mathcal{I}F,
\]
with zero initial conditions \(\mathcal{I}R(x,0)=0\), \(\partial_t (\mathcal{I}R)(x,0)=R(x,0)=0\), and zero boundary condition \(\mathcal{I}R|_\Sigma = 0\). By Theorem~\ref{thm wellp}, one obtains
\[
\|\mathcal{I}R\|_{L^\infty(0,T;H_0^1(\Omega))} + \|\partial_t (\mathcal{I}R)\|_{L^\infty(0,T;L^2(\Omega))} \le C \|\mathcal{I}F\|_{L^2(Q)}.
\]
Consequently,
\begin{equation}\label{eq rtr}
\|R\|_{L^2(Q)} \le C \|\mathcal{I}F\|_{L^2(Q)}, \quad
\|\partial_t R\|_{L^2(Q)} \le C \|F\|_{L^2(Q)}.
\end{equation}
Thus, to obtain the desired bounds for \(\|R\|_{L^2(Q)}\) and \(\|\partial_t R\|_{L^2(Q)}\) it remains to estimate \(\|\mathcal{I}F\|_{L^2(Q)}\) and \(\|F\|_{L^2(Q)}\), respectively.  Recalling that \(F = -p\partial_t^\alpha u_{\mathrm{app}} - q u_{\mathrm{app}}\), we have
\[
\mathcal{I}F = -\mathcal{I}(p\partial_t^\alpha u_{\mathrm{app}}) - \mathcal{I}(q u_{\mathrm{app}}), \quad
F = -p\partial_t^\alpha u_{\mathrm{app}} - q u_{\mathrm{app}}.
\]
We now estimate each term individually, starting with \(\mathcal{I}(q u_{\mathrm{app}})\).

\medskip
\noindent\textbf{Step 3. Estimate of \(\mathcal{I}(q u_{\rm app})\).}
Direct calculation gives
\[
\mathcal{I}(q u_{\mathrm{app}})(x,t) = q(x) \int_0^t e^{i\sigma\phi(x,\tau)} A(x,\tau)\,d\tau = q(x) e^{i\sigma x\cdot\omega} \int_0^t e^{i\sigma\tau} A(x,\tau)\,d\tau.
\]
Integrate by parts twice, using \(A(x,0)=0\) and \(A_t(x,0)=0\):
\[
\int_0^t e^{i\sigma\tau} A\,d\tau = \frac{e^{i\sigma t}}{(i\sigma)^2} A_t(x,t) - \frac{1}{(i\sigma)^2} \int_0^t e^{i\sigma\tau} A_{\tau\tau}\,d\tau.
\]
Consequently
\begin{equation}\label{eq quapp}
\|\mathcal{I}(q u_{\mathrm{app}})\|_{L^2(Q)} \le C\sigma^{-2}.
\end{equation}

\medskip

\noindent\textbf{Step 4. Estimate of \(\mathcal{I}(p\partial_t^\alpha u_{\mathrm{app}})\).}
Since \(u_{\mathrm{app}}(x,0)=0\), we have \(\mathcal{I}(\partial_t u_{\mathrm{app}}) = u_{\mathrm{app}}\). Using the commutativity of convolution and the integral operator,
\[
\mathcal{I}(p\partial_t^\alpha u_{\mathrm{app}}) = p \, \mathcal{I}(k_\alpha * \partial_t u_{\mathrm{app}}) = p \, k_\alpha * \mathcal{I}(\partial_t u_{\mathrm{app}}) = p \, k_\alpha * u_{\mathrm{app}}.
\]

Fix \(x \in \Omega\). By definition of \(u_{\mathrm{app}}\) and the change of variable \(s = t-\tau\),
\[
(k_\alpha * u_{\mathrm{app}})(x,t) = \int_0^t k_\alpha(t-\tau) e^{i\sigma(x\cdot\omega + \tau)} A(x,\tau) \, d\tau
= e^{i\sigma(x\cdot\omega + t)} \int_0^t k_\alpha(s) e^{-i\sigma s} A(x,t-s) \, ds.
\]
Denote the integral by
\[
J(t) = \int_0^t k_\alpha(s) e^{-i\sigma s} A(x,t-s) \, ds.
\]
Take \(\delta_1 = \sigma^{-1}\). If \(t\le \delta_1\), then \(|J(t)|\le C\int_0^{t} s^{-\alpha} ds = C\frac{t^{1-\alpha}}{1-\alpha}\le C\sigma^{\alpha-1}\). If \(t> \delta_1\), decompose the integral into \([0,\delta_1]\) and \([\delta_1, t]\).

Since \(|e^{-i\sigma s}| = 1\) and \(|A(x,t-s)| \le M := \sup|A|\),
\[
\begin{aligned}
\Bigl| \int_0^{\delta_1} k_\alpha(s) e^{-i\sigma s} A(x,t-s) \, ds \Bigr|
&\le M \int_0^{\delta_1} \frac{s^{-\alpha}}{\Gamma(1-\alpha)} \, ds
= \frac{M}{\Gamma(1-\alpha)} \cdot \frac{\delta_1^{1-\alpha}}{1-\alpha} \\
&\le \frac{M}{\Gamma(1-\alpha)(1-\alpha)} \sigma^{-(1-\alpha)} = C_1 \sigma^{\alpha-1}.
\end{aligned}
\]
On \([\delta_1, t]\), integration by parts gives
\[
\int_{\delta_1}^t k_\alpha(s) e^{-i\sigma s} A(x,t-s) \, ds
= \underbrace{\left[ \frac{e^{-i\sigma s}}{-i\sigma} k_\alpha(s) A(x,t-s) \right]_{s=\delta_1}^{s=t}}_{=:J_1}
+ \underbrace{\int_{\delta_1}^t \frac{e^{-i\sigma s}}{i\sigma} \frac{d}{ds}\big( k_\alpha(s) A(x,t-s) \big) \, ds}_{=:J_2}.
\]

For the boundary term \(J_1\): at \(s = t\) we have \(\frac{e^{-i\sigma t}}{-i\sigma} k_\alpha(t) A(x,0)=0\) because \(A(x,0)=0\). At \(s = \delta_1\) we have \(\frac{e^{-i\sigma \delta_1}}{-i\sigma} k_\alpha(\delta_1) A(x,t-\delta_1)\). Its magnitude is
\[
\frac{1}{\sigma} \cdot \frac{\delta_1^{-\alpha}}{\Gamma(1-\alpha)} \cdot M = \frac{M}{\Gamma(1-\alpha)} \sigma^{-1} \sigma^{\alpha} = C_2 \sigma^{\alpha-1}.
\]
Thus \(|J_1| \le C_1 \sigma^{\alpha-1}\) (with a generic constant \(C_1\)). For the integral term \(J_2\),
\[
\frac{d}{ds}\big( k_\alpha(s) A(x,t-s) \big) = k_\alpha'(s) A(x,t-s) - k_\alpha(s) A_t(x,t-s),
\]
where \(k_\alpha'(s) = -\alpha s^{-\alpha-1} / \Gamma(1-\alpha)\). Hence
\[
|J_2| \le \frac{1}{\sigma} \int_{\delta_1}^t \big( |k_\alpha'(s)| |A(x,t-s)| + |k_\alpha(s)| |A_t(x,t-s)| \big) \, ds.
\]
Let \(M_1 = \sup |A_t|\). Then
\[
\begin{aligned}
|J_2| &\le \frac{1}{\sigma} \left( \frac{\alpha M}{\Gamma(1-\alpha)} \int_{\delta_1}^t s^{-\alpha-1} \, ds + \frac{M_1}{\Gamma(1-\alpha)} \int_{\delta_1}^t s^{-\alpha} \, ds \right)\\
&=\frac{1}{\sigma} \left( \frac{\alpha M}{\Gamma(1-\alpha)} \frac{\delta_1^{-\alpha} - t^{-\alpha}}{\alpha} + \frac{M_1}{\Gamma(1-\alpha)}\frac{t^{1-\alpha} - \delta_1^{1-\alpha}}{1-\alpha}\right)\\
&\le \frac{M}{\Gamma(1-\alpha)} \sigma^{\alpha-1} + \frac{M_1 T^{1-\alpha}}{\Gamma(1-\alpha)(1-\alpha)} \sigma^{-1}.
\end{aligned}
\]
Since \(\sigma^{-1} \le \sigma^{\alpha-1}\) for large \(\sigma\), there exists a constant \(C_2\) such that \(|J_2| \le C_2 \sigma^{\alpha-1}\).

Collecting the above results, for every \(t \in [0,T]\),
\[
|(k_\alpha * u_{\mathrm{app}})(x,t)| = |e^{i\sigma(x\cdot\omega + t)} J(t)| = |J(t)| \le (C_1 + C_2) \sigma^{\alpha-1} = C_0 \sigma^{\alpha-1},
\]
where \(C_0=C_1+C_2\) is a constant. For fixed \(x\),
\[
\|(k_\alpha * u_{\mathrm{app}})(x,\cdot)\|_{L^2(0,T)}^2 = \int_0^T |(k_\alpha * u_{\mathrm{app}})(x,t)|^2 \, dt \le \int_0^T (C_0 \sigma^{\alpha-1})^2 \, dt = C_0^2 T \sigma^{2(\alpha-1)}.
\]
The bound is independent of \(x\); integrating over \(\Omega\) yields
\[
\|k_\alpha * u_{\mathrm{app}}\|_{L^2(Q)}^2 = \int_\Omega \|(k_\alpha * u_{\mathrm{app}})(x,\cdot)\|_{L^2(0,T)}^2 \, dx \le C_0^2 T |\Omega| \sigma^{2(\alpha-1)}.
\]
Taking the square root,
\begin{equation}\label{eq kuapp}
\|k_\alpha * u_{\mathrm{app}}\|_{L^2(Q)} \le C \sigma^{\alpha-1}, \quad C = C_0 \sqrt{T|\Omega|}.
\end{equation}
From \eqref{eq rtr}, \eqref{eq quapp} and \eqref{eq kuapp} we get
\[
\|R\|_{L^2(Q)} \le C\|\mathcal{I}F\|_{L^2(Q)} \le C \sigma^{\alpha-1}.
\]

\medskip
\noindent\textbf{Step 5. Estimate of \(\partial_t^\alpha u_{\rm app}\).}
Having obtained the desired bound for \(\|R\|_{L^2(Q)}\), we now turn to the estimate of \(\|\partial_t R\|_{L^2(Q)}\), which according to \eqref{eq rtr} is controlled by \(\|F\|_{L^2(Q)}\).  Since \(F = -p\partial_t^\alpha u_{\mathrm{app}} - q u_{\mathrm{app}}\) and \(|u_{\mathrm{app}}|\le 1\) implies \(\|q u_{\mathrm{app}}\|_{L^2(Q)}\le C\), the main task is to bound \(\|\partial_t^\alpha u_{\mathrm{app}}\|_{L^2(Q)}\). Recall that
\[
\partial_t u_{\mathrm{app}}(x,t) = e^{i\sigma(x\cdot\omega+t)}\big(i\sigma A(x,t) + A_t(x,t)\big)
=: e^{i\sigma x\cdot\omega} e^{i\sigma t} B(x,t),
\]
with \(B(x,t) = i\sigma A(x,t) + A_t(x,t)\).  The Caputo derivative is
\[
\partial_t^\alpha u_{\mathrm{app}}(x,t)
= \frac{1}{\Gamma(1-\alpha)}\int_0^t (t-s)^{-\alpha} \partial_t u_{\mathrm{app}}(x,s)\,ds
= \frac{e^{i\sigma x\cdot\omega}}{\Gamma(1-\alpha)}
\int_0^t (t-s)^{-\alpha} e^{i\sigma s} B(x,s)\,ds.
\]
We estimate the integral
\[
I(t) := \int_0^t (t-s)^{-\alpha} e^{i\sigma s} B(x,s)\,ds
\]
uniformly in \(x\) and \(t\). Note that \(B(x,0)=0\) because \(A(x,0)=0\) and \(A_t(x,0)=0\). The function \(B\) satisfies
\[
|B(x,s)| = |i\sigma A(x,s) + A_t(x,s)| \le \sigma|A(x,s)| + |A_t(x,s)| \le C\sigma
\]
for large \(\sigma\). Moreover,
\[
B_s(x,s) = \frac{\partial}{\partial s} \big( i\sigma A(x,s) + A_t(x,s) \big)
= i\sigma A_t(x,s) + A_{tt}(x,s).
\]

Since \(\theta_\epsilon\in C^\infty(\mathbb{R})\) is constant on
\((-\infty,0]\) and on \([\epsilon,\infty)\), its derivatives are supported in the transition interval \([0,\epsilon]\). In particular, \(A_t\) and \(A_{tt}\) are uniformly bounded, independently of \(\sigma\). Thus 
\[
|A_t(x,s)| \le C, \quad |A_{tt}(x,s)| \le C,
\]
and consequently
\[
|B_s(x,s)| \le \sigma|A_t(x,s)| + |A_{tt}(x,s)| \le C\sigma + C \le C\sigma
\]
for large \(\sigma\).

Set \(\delta_2 = \sigma^{-1}\). If \(t \le \delta_2\), then
\[
I(t) = \int_0^{t} (t-s)^{-\alpha} e^{i\sigma s} B(x,s)\,ds \le C\sigma \int_0^{t} (t-s)^{-\alpha}\,ds \le C\sigma \frac{t^{1-\alpha}}{1-\alpha} \le C\sigma^{\alpha}.
\]
If \(t > \delta_2\), we split the integral:
\[
I(t) = \int_0^{t-\delta_2} (t-s)^{-\alpha} e^{i\sigma s} B(x,s)\,ds
+ \int_{t-\delta_2}^t (t-s)^{-\alpha} e^{i\sigma s} B(x,s)\,ds
=: I_1 + I_2.
\]
For \(I_2\), since \(|B|\le C\sigma\),
\[
|I_2| \le C\sigma \int_{t-\delta_2}^t (t-s)^{-\alpha}\,ds = C\sigma \frac{\delta_2^{1-\alpha}}{1-\alpha} \le C\sigma^\alpha.
\]
On the interval \([0, t-\delta_2]\), integration by parts yields
\[
I_1 = \frac{1}{i\sigma}\Big[ (t-s)^{-\alpha} e^{i\sigma s} B(x,s) \Big]_{s=0}^{s=t-\delta_2}
- \frac{1}{i\sigma}\int_0^{t-\delta_2} e^{i\sigma s}\frac{\partial}{\partial s}\Big((t-s)^{-\alpha} B(x,s)\Big)\,ds.
\]
The boundary term at \(s=0\) vanishes because \(B(x,0)=0\). At \(s=t-\delta_2\), its modulus is
\[
\frac{1}{\sigma}\, \delta_2^{-\alpha} |e^{i\sigma(t-\delta_2)} B(x,t-\delta_2)|
\le \frac{1}{\sigma}\, \sigma^\alpha\, C\sigma = C\sigma^\alpha.
\]

For the remaining integral, compute
\[
\frac{\partial}{\partial s}\Big((t-s)^{-\alpha} B(x,s)\Big)
= \alpha (t-s)^{-\alpha-1} B(x,s) + (t-s)^{-\alpha} \partial_s B(x,s).
\]
Using \(|B|\le C\sigma\) and \(|B_s|\le C\sigma\),
\[
\begin{aligned}
\Bigl|\int_0^{t-\delta_2} e^{i\sigma s}\frac{\partial}{\partial s}\Big((t-s)^{-\alpha} B(x,s)\Big)\,ds\Bigr|
&\le \int_0^{t-\delta_2} \Big( \alpha (t-s)^{-\alpha-1} |B| + (t-s)^{-\alpha} |\partial_s B| \Big) ds\\
&\le C\sigma \int_0^{t-\delta_2} \Big( \alpha (t-s)^{-\alpha-1} + (t-s)^{-\alpha} \Big) ds.
\end{aligned}
\]
The first integral gives
\[
\int_0^{t-\delta_2} (t-s)^{-\alpha-1} ds = \frac{1}{\alpha}\big(\delta_2^{-\alpha} - t^{-\alpha}\big) \le C\sigma^{\alpha},
\]
and the second is
\[
\int_0^{t-\delta_2} (t-s)^{-\alpha} ds = \frac{1}{1-\alpha}\big(t^{1-\alpha} - \delta_2^{1-\alpha}\big) \le C.
\]
Hence the whole expression is bounded by \(C\sigma(\sigma^{\alpha}+1)\). Multiplying by \(\frac{1}{\sigma}\) yields
\[
\frac{1}{\sigma} \cdot C\sigma(\sigma^{\alpha}+1) \le C(\sigma^\alpha+1) \le C\sigma^\alpha \quad (\sigma\ge 1).
\]
Collecting all parts, we have \(|I_1| \le C\sigma^\alpha\) and \(|I_2| \le C\sigma^\alpha\). Thus 
\[
|\partial_t^\alpha u_{\mathrm{app}}(x,t)| \le \frac{C}{\Gamma(1-\alpha)} \sigma^\alpha,
\]
and integrating over \(Q\) gives \(\|\partial_t^\alpha u_{\mathrm{app}}\|_{L^2(Q)} \le C\sigma^\alpha\).

With these estimates, we obtain
\[
\|\partial_t R\|_{L^2(Q)} \le C\|F\|_{L^2(Q)} \le C\bigl(\|p\|_{L^\infty(\Omega)}\|\partial_t^\alpha u_{\mathrm{app}}\|_{L^2(Q)} + \|q u_{\mathrm{app}}\|_{L^2(Q)}\bigr) \le C\sigma^\alpha.
\]
Together with the bound \(\|R\|_{L^2(Q)} \le C\sigma^{\alpha-1}\) already established in Step~4, this completes the proof of Lemma~\ref{lem cgo}.
\end{proof}

\begin{proof}[\textbf{Proof of Theorem~\ref{thm uniqueness}}]

The proof proceeds in five steps. We first introduce the high-frequency approximations and insert them into the integral identity. We then estimate the error terms \(\mathcal{E}_p(t)\) and \(\mathcal{E}_q(t)\). After that, we compute the leading asymptotics of the principal terms. Finally, we extract the Fourier transforms of the coefficient differences \(p\) and \(q\), and show that they vanish identically.

\medskip

\noindent\textbf{Step 1. CGO construction and expansion of the identity.}
Following the strategy outlined above, we construct high-frequency beam solutions adapted to the fractional wave operator and substitute them into the integral identity \eqref{eq iduv0}. 

Let \(\xi\in\mathbb{R}^n\) be fixed. For sufficiently large \(\sigma>0\), choose unit vectors \(\omega_1,\omega_2\in\mathbb{S}^{n-1}\) satisfying \(\omega_1+\omega_2=-\frac{\xi}{\sigma}\). This is possible whenever \(\sigma>|\xi|/2\). Define the beam solutions
\[
u_{\rm app}(x,t)=e^{i\sigma\phi_1(x,t)}A_1(x,t),
\qquad
v_{\rm app}(x,t)=e^{i\sigma\phi_2(x,t)}A_2(x,t),
\]
where
\[
\phi_j(x,t)=x\cdot\omega_j+t,
\qquad
A_j(x,t)=\theta_\epsilon(\phi_j(x,t)-r),
\qquad j=1,2.
\]
Let \(u\) and \(v\) be the corresponding exact solutions of \eqref{eq1.1} with boundary data \(u_{\rm app}|_\Sigma\), \(v_{\rm app}|_\Sigma\),
respectively, and with zero initial conditions. Set
\[
R_u=u-u_{\rm app},
\qquad
R_v=v-v_{\rm app}.
\]
By Lemma~\ref{lem cgo}, these remainders satisfy
\[
\|R_u\|_{L^2(Q)} + \|R_v\|_{L^2(Q)} \le C \sigma^{\alpha-1}, \qquad
\|\partial_t R_u\|_{L^2(Q)} + \|\partial_t R_v\|_{L^2(Q)} \le C \sigma^{\alpha},
\]
and vanish on the lateral boundary: \(R_u|_{\Sigma}=R_v|_{\Sigma}=0\).

Substituting \(u = u_{\mathrm{app}} + R_u\), \(v = v_{\mathrm{app}} + R_v\) into the integral identity \eqref{eq iduv0} yields
\begin{equation}\label{eq uvespesp}
\int_{\Omega} p \, k_\alpha * \partial_t (u_{\mathrm{app}}*v_{\mathrm{app}}) \, dx + \int_{\Omega} q \, u_{\mathrm{app}}*v_{\mathrm{app}} \, dx + \mathcal{E}_p(t)+ \mathcal{E}_q(t) = 0,    
\end{equation}
where
\[
\mathcal{E}_p(t)
:=
\int_\Omega p\,k_\alpha*\partial_t
\bigl(
u_{\rm app}*R_v
+
R_u*v_{\rm app}
+
R_u*R_v
\bigr)\,dx,
\]
and
\[
\mathcal{E}_q(t)
:=
\int_\Omega q
\bigl(
u_{\rm app}*R_v
+
R_u*v_{\rm app}
+
R_u*R_v
\bigr)\,dx.
\]

\medskip

\noindent\textbf{Step 2. Estimate of \(\mathcal{E}_p(t)\).}
We first estimate the error term \(\mathcal{E}_p(t)\). Consider the term involving \(u_{\rm app}*R_v\). Since \(u_{\rm app}\) and \(R_v\) vanish at \(t=0\), the differentiation rule for convolution gives
\[
\partial_t(u_{\rm app}*R_v)
=
(\partial_tu_{\rm app})*R_v
=
u_{\rm app}*\partial_tR_v.
\]
In particular, we use
\[
k_\alpha*\partial_t(u_{\rm app}*R_v)
=
(k_\alpha*\partial_tu_{\rm app})*R_v.
\]

For the fixed \(x\in\Omega\), by the
Cauchy-Schwarz inequality and Young inequality,
\begin{equation*}
\begin{aligned}
|(k_{\alpha}*\partial_t u_{\rm app})*R_v|
&\le \int_0^t |(k_\alpha*\partial_t u_{\rm app})(x,t-\tau) R_v(x,\tau)|\,d\tau\\
&\le \|k_{\alpha}*\partial_t u_{\rm app}(x,\cdot)\|_{L^2(0,T)} \|R_v(x,\cdot)\|_{L^2(0,T)}\\[2pt]
&\le C \|\partial_t u_{\rm app}(x,\cdot)\|_{L^2(0,T)} \|R_v(x,\cdot)\|_{L^2(0,T)}.
\end{aligned}    
\end{equation*}
According to Lemma \ref{lem cgo}, one obtains \(\|\partial_t u_{\rm app}\|_{L^2(Q)}\le C\sigma^{\alpha}\) and  \(\|R_v\|_{L^2(Q)}\le C\sigma^{\alpha-1}\). Integrating over \(\Omega\) and using the Cauchy-Schwarz inequality yields
\[
\begin{aligned}
\int_\Omega |(k_{\alpha}*\partial_t u_{\rm app})*R_v|\,dx
&\le C\int_\Omega \| \partial_t u_{\rm app}(x,\cdot) \|_{L^2(0,T)} \| R_v(x,\cdot) \|_{L^2(0,T)} \,dx \\
&\le C \| \partial_t u_{\rm app} \|_{L^2(Q)} \| R_v \|_{L^2(Q)}\le C \sigma^\alpha \cdot \sigma^{\alpha-1} = C \sigma^{2\alpha-1},
\end{aligned}
\]
where \(C\) is independent of \(t\) and \(\sigma\). Both the term with \(R_u*v_{\rm app}\) and the term with \(R_u*R_v\) are handled in exactly the same way, yielding
\begin{equation*}\label{eq espp2}
\Bigl|\int_\Omega p\,k_\alpha*\partial_t(R_u*v_{\rm app})\,dx\Bigr|\le C\sigma^{2\alpha-1}, \quad
\Bigl|\int_\Omega p\,k_\alpha*\partial_t(R_u*R_v)\,dx\Bigr|
\le C\sigma^{2\alpha-1}.
\end{equation*}
Combining the above results, we obtain \(|\mathcal{E}_p(t)| \le C \sigma^{2\alpha-1}\) for \(t \in [0,T]\).

\medskip

\noindent\textbf{Step 3.  Estimate of \(\mathcal{E}_q(t)\).}
Next, we estimate the error terms \( \mathcal{E}_q(t)\). For \(t\in[0,T]\), using the definition of convolution and the bound \(|u_{\mathrm{app}}|\le C\), we get
\[
|u_{\mathrm{app}}*R_v|
\le C\int_0^t |R_v(x,\tau)|\,d\tau
\le C\sqrt{T}\,\|R_v(x,\cdot)\|_{L^2(0,T)} .
\]
Hence,
\[
\begin{aligned}
\Bigl|\int_\Omega q \,(u_{\mathrm{app}}*R_v)\,dx\Bigr|
&\le \|q\|_{L^\infty(\Omega)} \int_\Omega |(u_{\mathrm{app}}*R_v)|\,dx \\
&\le C \int_\Omega \|R_v(x,\cdot)\|_{L^2(0,T)}\,dx \\
&\le  C \|R_v\|_{L^2(Q)} \le C\sigma^{\alpha-1}.
\end{aligned}
\]
where the constant does not depend on \(t\) and \(\sigma\). The same argument applies to the term \(R_u*v_{\mathrm{app}}\), and one obtains
\[
\Bigl| \int_{\Omega} q \, R_u*v_{\mathrm{app}} \, dx\Bigr| \le C \sigma^{\alpha-1}.
\]
For the term involving \(R_u*R_v\), by the Cauchy–Schwarz inequality,
\[
\begin{aligned}
\Bigl|\int_\Omega q \,(R_u*R_v)\,dx\Bigr|
&\le \|q\|_{L^\infty(\Omega)} \int_\Omega |(R_u*R_v)|\,dx \\
&\le C \int_\Omega \|R_u(x,\cdot)\|_{L^2(0,T)} \|R_v(x,\cdot)\|_{L^2(0,T)}\,dx \\
&\le C\|R_u\|_{L^2(Q)} \|R_v\|_{L^2(Q)}\le C \sigma^{\alpha-1}\cdot\sigma^{\alpha-1} \le C\sigma^{\alpha-1}
\end{aligned}
\]
for large \(\sigma\). Combining the above results, \(|\mathcal{E}_q (t)|\leq C\sigma^{\alpha-1}\) for \(t\in [0,T]\).

\medskip

\noindent\textbf{Step 4. Asymptotic expansion of the principal terms.} With the error terms \(\mathcal{E}_p(t)\) and \(\mathcal{E}_q(t)\) bounded as in Steps~2 and~3, we now turn to the asymptotic analysis of the principal terms in \eqref{eq uvespesp}. According to the definition of the beam solutions
\[
u_{\rm app} = e^{i\sigma(x\cdot\omega_1 + t)}A_1(x,t),\quad
v_{\rm app} = e^{i\sigma(x\cdot\omega_2 + t)}A_2(x,t),
\]
with \(\omega_1+\omega_2 = -\xi/\sigma\). A direct computation of the convolution gives
\begin{equation}\label{eq uv}
u_{\rm app}*v_{\rm app}(x,t) = e^{-i\xi\cdot x} e^{i\sigma t} \gamma(x,t),
\end{equation}
where
\[
\gamma(x,t) = \int_0^t A_1(x,t-\tau)A_2(x,\tau)\,d\tau.
\]
Differentiating in \(t\) yields
\[
\partial_t(u_{\rm app}*v_{\rm app}) = e^{-i\xi\cdot x} e^{i\sigma t} \big(i\sigma \gamma(x,t) + \gamma_t(x,t)\big),
\]
and hence
\[
k_\alpha*\partial_t(u_{\rm app}*v_{\rm app})
= e^{-i\xi\cdot x} \int_0^t k_\alpha(t-\tau) e^{i\sigma \tau} \big(i\sigma \gamma(x,\tau) + \gamma_t(x,\tau)\big) d\tau.
\]
Changing variables \(s = t - \tau\), we write
\[
k_\alpha*\partial_t(u_{\rm app}*v_{\rm app}) = e^{-i\xi\cdot x} e^{i\sigma t} h(t), \quad
h(t) = \int_0^t k_\alpha(s) e^{-i\sigma s} f(s)\,ds,
\]
with
\[
f(s) = i\sigma \gamma(x,t-s) + \gamma_t(x,t-s).
\]

For \(t \in (2r+2\epsilon,T)\), the cut-off amplitudes satisfy \(A_1 = A_2 = 1\) in \(\Omega\), so that
\[
\gamma(x,t-s) = (t-s) - t_0(x), \quad \gamma_t(x,t-s) = 1,
\]
where \(t_0(x)\) is a bounded function independent of \(t\). Therefore,
\[
f(s) = i\sigma ((t-s) - t_0(x)) + 1 = i\sigma(t - t_0(x)) + 1 - i\sigma s.
\]

To evaluate the integrals asymptotically for large \(\sigma\), we extend the upper limit to infinity and control the tail by integration by parts.  For any \(\beta>-1\) and \(t>0\),
\[
\int_0^t s^{\beta} e^{-i\sigma s}\,ds
= \int_0^{\infty} s^{\beta} e^{-i\sigma s}\,ds - \int_t^{\infty} s^{\beta} e^{-i\sigma s}\,ds
= \Gamma(\beta+1)(i\sigma)^{-\beta-1} - \int_t^{\infty} s^{\beta} e^{-i\sigma s}\,ds.
\]
For the tail integral we integrate by parts once:
\[
\int_t^{\infty} s^{\beta} e^{-i\sigma s}\,ds
= \frac{t^{\beta} e^{-i\sigma t}}{i\sigma} + \frac{\beta}{i\sigma}\int_t^{\infty} s^{\beta-1} e^{-i\sigma s}\,ds
= O(\sigma^{-1}),
\]
since both terms on the right‑hand side are bounded by \(C\sigma^{-1}\) for \(\sigma\ge 1\) (the remaining integral converges absolutely because \(\beta-1>-2\) and the oscillatory factor ensures integrability at infinity).  Consequently,
\[
\int_0^t s^{\beta} e^{-i\sigma s}\,ds
= \Gamma(\beta+1)(i\sigma)^{-\beta-1} + O(\sigma^{-1}) \qquad \text{for } \beta>-1.
\]
Applying this with \(\beta = -\alpha\) and \(\beta = 1-\alpha\) gives
\[
\int_0^t k_\alpha(s) e^{-i\sigma s}\, ds = (i\sigma)^{\alpha-1} + O(\sigma^{-1}),\quad
\int_0^t s k_\alpha(s) e^{-i\sigma s}\, ds = (1-\alpha)(i\sigma)^{\alpha-2} + O(\sigma^{-1}).
\]
Substituting these expansions into the expression for \(h(t)\) yields
\[
\begin{aligned}
h(t) &= \big(i\sigma(t-t_0) + 1\big) \int_0^t k_\alpha(s) e^{-i\sigma s}\, ds - i\sigma \int_0^t s k_\alpha(s) e^{-i\sigma s}\, ds \\
&= (i\sigma)^\alpha (t-t_0) + \alpha (i\sigma)^{\alpha-1} + O(\sigma^{-1}).
\end{aligned}
\]
Finally, multiplying back by \(e^{-i\xi\cdot x} e^{i\sigma t}\) gives the asymptotic expansion, we arrive at
\begin{equation}\label{eq paruparv}
\begin{aligned}
k_\alpha*\partial_t(u_{\rm app}*v_{\rm app})(x,t)
= e^{-i\xi\cdot x}e^{i\sigma t}\Big[ (i\sigma)^\alpha(t-t_0(x)) + \alpha(i\sigma)^{\alpha-1} \Big]
+ O(\sigma^{-1}).
\end{aligned}
\end{equation}

\medskip

\noindent\textbf{Step 5.  Extraction of the Fourier coefficients and conclusion.}
With the asymptotic expansion of the principal terms established in Step~4 and the error bounds obtained in Steps~2--3, we are now in a position to substitute everything into the identity \eqref{eq uvespesp} and extract the Fourier coefficients of \(p\) and \(q\).

Substituting \eqref{eq uv} and \eqref{eq paruparv} into the identity \eqref{eq uvespesp} yields 
\begin{equation}\label{eq substituted}
\begin{aligned}
e^{i\sigma t}\int_\Omega e^{-i\xi\cdot x}
\Big[ p\big( (i\sigma)^\alpha(t-t_0) + \alpha(i\sigma)^{\alpha-1} \big)+ q (t-t_0) \Big]dx + \mathcal{E}_p(t) + \mathcal{E}_q(t) + O(\sigma^{-1}) = 0
\end{aligned}
\end{equation}
for all \(t\in (2r+2\epsilon,T)\), the error terms satisfy
\[
|\mathcal{E}_p(t)| \le C\sigma^{2\alpha-1}, \quad
|\mathcal{E}_q(t)| \le C\sigma^{\alpha-1},
\]
where the constants \(C\) are independent of \(\sigma\) and \(t\). 

Introduce the notation
\[
\hat p(\xi) = \int_\Omega p(x) e^{-i\xi\cdot x}dx,\quad
\hat q(\xi) = \int_\Omega q(x) e^{-i\xi\cdot x}dx,
\]
\[P(\xi) = \int_\Omega p(x) e^{-i\xi\cdot x}t_0(x)dx,\quad
Q(\xi) = \int_\Omega q(x) e^{-i\xi\cdot x}t_0(x)dx.
\]
Then \eqref{eq substituted} can be rewritten as
\begin{equation}\label{eq compact}
e^{i\sigma t}\Big[ (i\sigma)^\alpha\big(t\hat p - P\big) + \alpha(i\sigma)^{\alpha-1}\hat p + \big(t\hat q - Q\big) \Big]
= -\mathcal{E}_p(t)- \mathcal{E}_q(t) + O(\sigma^{-1}).
\end{equation}

We first recover \(p\). Dividing \eqref{eq compact} by \(\sigma^\alpha\), and letting \(\sigma\to\infty\). Since 
\[
\frac{|\mathcal{E}_p(t)|}{\sigma^\alpha} \le C\sigma^{\alpha-1} \rightarrow 0,\,\,\,\,
\frac{|\mathcal{E}_q(t)|}{\sigma^\alpha} \le C\sigma^{-1} \to 0, \,\,\,\,\frac{O(\sigma^{-1})}{\sigma^\alpha} = O(\sigma^{-\alpha-1}) \to 0,
\]
the right-hand side of \eqref{eq compact} tends to zero.  Multiplying by \(e^{-i\sigma t}\) and taking the limit \(\sigma\to\infty\),
\[
i^\alpha\big(t\hat p(\xi) - P(\xi)\big) = 0 \quad t\in(2r+2\epsilon,T).
\]
Taking two distinct times \(t_1,t_2\in(2r+2\epsilon,T)\) and subtracting the corresponding identities. Because \(P\) does not depend on \(t\), it cancels and we get
\[
i^\alpha(t_1-t_2)\hat p(\xi)=0.
\]
Since \(t_1\ne t_2\), it follows that \(\hat p(\xi)=0\). As \(\xi\in\mathbb{R}^n\) was arbitrary, the Fourier transform of the zero extension of \(p\) vanishes identically. Hence \(p=0\) a.e. in \(\Omega\).

Having proved \(p=0\), all terms containing \(p\) disappear from the original identity.  In particular, \(\mathcal{E}_p(t)\equiv 0\) and the principal terms involving \(p\) vanish.  \eqref{eq compact} reduces to
\[
e^{i\sigma t}\big(t\hat q(\xi) - Q(\xi)\big) = -\mathcal{E}_q(t) + O(\sigma^{-1}).
\]
Multiplying by \(e^{-i\sigma t}\) and using the bound \(|\mathcal{E}_q(t)|\le C\sigma^{\alpha-1}\) gives
\[
t\hat q(\xi) - Q(\xi) = -e^{-i\sigma t}\mathcal{E}_q(t) + O(\sigma^{-1})
= O(\sigma^{\alpha-1}) + O(\sigma^{-1}) = O(\sigma^{\alpha-1}).
\]
Taking again two distinct times \(t_3,t_4\in(2r+2\epsilon,T)\) and subtracting the two identities; the term \(Q\) cancels and let \(\sigma\to\infty\),
\[
(t_3-t_4)\hat q(\xi) = O(\sigma^{\alpha-1}) \rightarrow 0,
\]
which forces \(\hat q(\xi) = 0\) for all \(\xi\in\mathbb{R}^n\). The Fourier transform of the zero-extension of \(q\) is zero, yielding \(q=0\) a.e. in \(\Omega\). This completes the proof of Theorem \ref{thm uniqueness}.
\end{proof}


\section{Concluding remarks}\label{sec4}

In this work, we first established the well-posedness of the forward problem for boundary data that are sufficiently smooth to allow the construction of a suitable lifting function. Building on this foundation, we investigated the inverse problem of determining the two spatially dependent coefficients \(p(x)\) and \(q(x)\) in a time-fractional damped wave equation from the DtN map. By constructing high-frequency beam solutions concentrating along straight lines and exploiting a convolution-type integral identity that eliminates the unknown boundary traces, we proved that the DtN map uniquely determines both coefficients. The argument combines the well-posedness theory, the construction of such special solutions, and a careful remainder estimate based on oscillatory integration by parts to control the error of the geometric-optics approximation. Consequently, equality of the DtN maps implies that the Fourier transforms of the coefficient differences vanish identically, yielding \(p_1=p_2\) and \(q_1=q_2\) almost everywhere in \(\Omega\).

Our results extend classical uniqueness theorems for hyperbolic inverse problems to the fractional-order setting, providing a rigorous theoretical foundation for applications where memory and viscoelastic effects are significant. Natural directions for future research include extensions to anisotropic coefficients, multi-term fractional equations, partial boundary measurements, and quantitative stability analysis of the inverse map.


\section*{Acknowledgements}
The author would like to express their sincere gratitude to Professors Yavar Kian and Zhiyuan Li for their invaluable guidance and helpful suggestions throughout the preparation of this work.


\bibliographystyle{unsrt}

\end{document}